\documentclass[11pt]{article}

\usepackage{fullpage,amsmath,amssymb,amsthm}
\usepackage{tikz}
\usepackage{graphics}

\newtheorem{theorem}{Theorem}
\newtheorem{prop}[theorem]{Proposition}

\newtheorem{cor}[theorem]{Corollary}
\newtheorem{lemma}[theorem]{Lemma}

\newtheorem{coloring}[theorem]{Coloring}
\newtheorem{problem}[theorem]{Problem}

\numberwithin{theorem}{section}

\newcommand{\f}{\overline{f}}
\newcommand{\ex}{\text{ex}}
\newcommand{\ep}{\varepsilon}
\newcommand{\floor}[1]{\left\lfloor #1 \right\rfloor}
\newcommand{\ceil}[1]{\left\lceil #1 \right\rceil}

\title{Generalized Ramsey numbers: forbidding paths with few colors}
\author{Robert A. Krueger\footnote{Department of Mathematics, Miami University, \texttt{kruegera@miamioh.edu}. Research was conducted at the REU at the City University of New York's Baruch College, which was supported by NSF grant DMS-1710305.}}
\date{June 16, 2019}

\begin{document}

\maketitle

\begin{abstract}

Let $f(K_n, H, q)$ be the minimum number of colors needed to edge-color $K_n$ so that every copy of $H$ is colored with at least $q$ colors. Originally posed by Erd\H{o}s and Shelah when $H$ is complete, the asymptotics of this extremal function have been extensively studied when $H$ is a complete graph or a complete balanced bipartite graph. Here we investigate this function for some other $H$, and in particular we determine the asymptotic behavior of $f(K_n, P_v, q)$ for almost all values of $v$ and $q$, where $P_v$ is a path on $v$ vertices.

\end{abstract}

\section{Introduction}





The classical graph Ramsey problem asks, ``for a given $k$ and $p$, what is the smallest $n$ such that every edge-coloring of $K_n$ with $k$ colors contains a monochromatic $K_p$?'' We may invert the roles of $k$ and $n$ in this question: for a given $n$ and $p$, what is the largest $k$ such that every edge-coloring of $K_n$ with $k$ colors contains a monochromatic $K_p$? This is equivalent to determining the smallest $k+1$ such that there exists an edge-coloring of $K_n$ with $k+1$ colors such that every copy of $K_p$ contains at least $2$ colors. Let $f(G, H, q)$ be the minimum number of colors needed to edge-color $G$ so that every copy of $H$ in $G$ contains at least $q$ colors.\footnote{In the literature these are called $(H, q)$-colorings, and $G$ is typically taken to be $K_n$ or $K_{n,n}$ to ensure no difficulty in finding copies of $H$; see for example \cite{axenovich}.} The classical Ramsey problem is then equivalent to determining $f(K_n, K_p, 2)$; as an example from \cite{gyarfas}, the Ramsey number $R(3,3,3) = 17$ is equivalent to $f(K_{16}, K_3, 2) = 3$ and $f(K_{17}, K_3, 2) = 4$. Additionally, $f(G, H, e(H))$ can be viewed as a variant of the rainbow Ramsey problem, where every copy of $H$ must receive all distinct colors (as defined in \cite{rainbow}). In this way studying $f(G, H, q)$ for general $q$ bridges the gap between these problems.

The function $f(K_n, K_p, q)$ was first introduced by Erd\H{o}s and Shelah (see Section 5 of \cite{earlyerdos}), and while Elekes, Erd\H{o}s, and F\"uredi had some preliminary results (described in Section 9 of \cite{erdos}), the problem was first systematically studied by Erd\H{o}s and Gy\'arf\'as \cite{gyarfas}. They focused on the asymptotics of $f(K_n, K_p, q)$ for large $n$, with $p$ and $q$ fixed. They also determined the \textbf{linear and quadratic thresholds} of this function, that is, they determined the smallest $q(p)$ such that $f(K_n, K_p, q(p)) = \Omega(n)$ (or $=\Omega(n^2)$, resp.). Axenovich, F\"uredi, and Mubayi \cite{axenovich} adapted these results to $f(K_{n,n}, K_{p,p}, q)$ in addition to relating these problems to other results in extremal combinatorics. Many others (only a few of which are mentioned here) have studied $f(K_n, K_p, q)$ and $f(K_{n,n}, K_{p,p}, q)$ for $q$ between the linear and quadratic thresholds \cite{sarklin,sarkbiplin}, above the quadratic threshold \cite{axenovich,sarkquad}, and at the `polynomial' threshold \cite{poly}. Besides two general results in \cite{axenovich}, little has been said for $f(K_n, H, q)$ when $H$ is not complete or complete bipartite.

The aim of this paper is to open up the study of $f(K_n, H, q)$ for general $H$. We consider $H$ and $q$ as fixed, determining the asymptotics of $f(K_n, H, q)$ in terms of $n$. In Section 2 we make some general observations for all $H$, supplementing those in \cite{axenovich}. One of the first Ramsey results for non-complete graphs is due to Gerencs\'er and Gy\'arf\'as \cite{gg}, who determined that every 2-coloring of $K_n$ has a monochromatic path of order $\ceil{(2n+1)/3}$. Inspired by this, in Section 3 we find the asymptotics of $f(K_n, P_v, q)$ for most $v$ and $q$, where $P_v$ denotes the path on $v$ vertices:

\begin{theorem}\label{maintheorem}
Let $v \geq 3$. The smallest $q(v)$ for which $f(K_n, P_v, q(v)) = \Theta(n^2)$ is $q(v) = \ceil{v/2}+1$. The largest $q(v)$ for which $f(K_n, P_v, q(v)) = \Theta(n)$ is $q(v) = \floor{v/2}$, unless $v=3$ or $v=5$, in which case $q(3) = 2$ and $q(5)=3$.
\end{theorem}

\begin{figure}
\centering
\caption{Bounds on $f(K_n, P_v, q)$}
\label{results}
\resizebox{\columnwidth}{!}{%
\begin{tabular}{|c|c|c|c|c|} \hline
& $v$ even & $v=3,5$ & $v=7$ & $v\geq 9$, odd \\ \hline
$2\leq q \leq \floor{v/2}$ & $\Theta(n)$ & $\Theta(n)$ & $\Theta(n)$ & $\Theta(n)$ \\ \hline
$q = \frac{v+1}{2}$ & $-$ & $\Theta(n)$ & $\Omega(n^{4/3}/\log^{2/3} n)$, $O(n^{5/3})$ & $\Omega(n^{3/2}/\log n)$, $O(n^{2-2/(v-1)})$ \\ \hline
$\ceil{v/2}+1 \leq q \leq v-1$ & $\Theta(n^2)$ & $\Theta(n^2)$ & $\Theta(n^2)$ & $\Theta(n^2)$ \\ \hline
\end{tabular}%
}
\end{figure}

See Figure \ref{results} for a complete summary of the results of Section 3. The only undetermined cases are when $v \geq 7$ is odd and $q = \frac{v+1}{2}$. Here we can show that $f(K_n, P_v, q)$ is neither linear nor quadratic in $n$, but somewhere in between. It is not clear what the correct asymptotics are for these cases, so we pose the following problem.

\begin{problem}
Asymptotically determine $f(K_n, P_v, \frac{v+1}{2})$ for odd $v \geq 7$.
\end{problem}

Note that $f(K_n, P_v, q)$ is almost always either $\Theta(n)$ or $\Theta(n^2)$ (with the only notable exception mentioned above). This is in great contrast to the behavior of $f(K_n, K_p, q)$. Pohoata and Sheffer \cite{adam} showed that for $p \geq 2(m+1) \geq 6$,
\[ f\left(K_n, K_p, \binom{p}{2} - m\cdot \floor{\frac{p}{m+1}} + m + 1\right) = \Omega\left(n^{1+\frac{1}{m}}\right) ,\]
and they noted that Theorem \ref{LLL} (below) implies an upper bound of
\[ f\left(K_n, K_p, \binom{p}{2} - m\cdot \floor{\frac{p}{m+1}} + m + 1\right) = O\left( n^{1 + \frac{1}{m} + \ep(p)} \right) ,\]
where $\ep(p)$ goes to zero as $p$ grows. Thus, for each value of $m$, the upper bound gets asymptotically close to the lower bound for sufficiently large $p$. This means that $f(K_n, K_p, q)$, as a function of $q$, attains arbitrarily many asymptotic values (even while ignoring subpolynomial factors), for sufficiently large $p$. Similarly tight results can be obtained for the bipartite variant $f(K_{n,n}, K_{p,p}, q)$, using essentially the same method as \cite{adam}.

Define $T(H)$, the number of `tiers' for $H$, to be the number of $q \geq 2$ such that $f(K_n, H, q)$ and $f(K_n, H, q+1)$ differ by some polynomial factor. Above we concluded that $T(K_p)$ is some unbounded function of $p$, whereas Theorem \ref{maintheorem} and the results of Section 2 show that $T(P_v), T(S_t), T(tK_2)$ are all at most $3$ for all $v$ and $t$, where $S_t$ and $tK_2$ is a star and a matching on $t$ edges, respectively. It would be interesting to know for what other classes $\mathcal{H}$ of graphs is $T(H)$ bounded for all $H \in \mathcal{H}$. This may be the case when the graphs of $\mathcal{H}$ are sufficiently sparse, perhaps if they were all trees.

\subsection{Notation and Terminology}

The number of edges of $H$ is denoted by $e(H)$, while the number of vertices of $H$ is denoted by $v(H)$. A copy of $H$ in $G$ is a subgraph of $G$ isomorphic to $H$. All colorings are edge-colorings. For a color $c$, the $c$-degree of a vertex $u$, denoted $d_c(u)$, is the number of edges of color $c$ incident to $u$. The Tur\'an number of $H$, denoted $\ex(n, H)$ is the largest number of edges an $n$-vertex graph may have without containing a copy of $H$. For a positive integer $t$, $tH$ is the disjoint union of $t$ copies of $H$.

We employ the following asymptotic notation throughout: $f(n) = O(g(n))$ means that there exists $c>0$ and $n_0$ such that $f(n) \leq c g(n)$ for all $n \geq n_0$. Similarly, $f(n) = \Omega(g(n))$ corresponds to $f(n) \geq c g(n)$. If $f(n) = O(g(n))$ and $f(n) = \Omega(g(n))$, then we say $f(n) = \Theta(g(n))$. All logarithms are base $2$.

It will often be helpful to think of $f(G, H, q)$ in terms of \textbf{repeated colors}. If color $c$ appears on exactly $r+1$ edges in a coloring of $H$, then we say color $c$ is repeated $r$ times. We say $H$ has $r$ `repetitions' or `repeats' if $r=\sum r_i$, where color $c_i$ is repeated $r_i$ times, the sum being over all colors. Note that if $H$ has $r$ repetitions and contains exactly $q$ colors, then $r+q = e(H)$. With this in mind, let $\f(G, H, r) = f(G, H, e(H)-r)$. For example, the number of colors required so that there are no `isosceles' triangles is $\f(K_n, K_3, 0) = f(K_n, K_3, 3)$.

\section{General Observations}

A wide-reaching upper bound on $f(G, H, q)$ is achieved with a simple application of the Lov\'asz Local Lemma. This idea first appeared in \cite{gyarfas} but was stated in full generality in \cite{axenovich}. The following bound is close to asymptotically tight in many cases when $q$ is large.

\begin{theorem}\label{LLL}
Let $v(H) = v$ and $e(H) = e$. Then $f(K_n, H, q) = O\left(n^{\frac{v-2}{e-q+1}}\right)$, or equivalently $\f(K_n, H, r) = O\left(n^{\frac{v-2}{r+1}}\right)$.
\end{theorem}

In the tradition of Erd\H{o}s and Gy\'arf\'as \cite{gyarfas}, we ask what are the linear and quadratic thresholds for any graph $H$. The question of the linear threshold was solved for connected $H$ in \cite{axenovich}, whose result and proof we give here in slightly more generality for completeness.



\begin{theorem}\label{lin}
Let $c$ be the number of connected components of $H$, let $v(H) = v$, and let $e(H) = e$. Then $f(K_n, H, e-v+2+c) = \Omega(n)$ and $f(K_n, H, e-v+2) = O\left(n^{1-\frac{1}{v-1}}\right)$.
\end{theorem}

\begin{proof}
The upper bound comes from Theorem \ref{LLL}.

For the lower bound, we show that if one color class is too large, we can find a copy of $H$ with at least $v-1-c$ repeats. Let $F$ be an edge-maximal spanning forest of $H$ (since $H$ may not be connected), and let $T$ be a tree on $v$ vertices which contains $F$. Note that $T$ consists of $F$ along with $c-1$ more edges. It is well known that $\ex(n, T) \leq vn$, that is, any graph on $n$ vertices with at least $vn$ edges contains every tree on $v$ vertices.\footnote{Indeed, a graph on $n$ vertices with $vn$ edges has average degree $2v$. Upon repeatedly deleting the vertices of degree less than $v$, we have removed fewer than $vn$ edges, and so we have a nonempty graph of minimum degree at least $v$, into which we can embed $T$. In fact, the Erd\H{o}s-S\'os conjecture states that $\ex(n, T) \leq (v-2)n/2$, which Ajtai, Koml\'os, Simonovits, and Szemer\'edi have shown to be true for sufficiently large $v$ and $n$ (unpublished, see e.g., Section 6 of \cite{trees}).} If some color in a coloring of $K_n$ contains more than $vn$ edges, then we can find a monochromatic $T$ and thus a monochromatic $F$, which has $v-c-1$ repeats. This in turn gives a copy of $H$ with at least $v-c-1$ repeats. If we require each copy of $H$ to have at most $v-c-2$ repeats, then we must use at least $\binom{n}{2} / (vn)$ colors. More succinctly,
\[ \f(K_n, H, v-c-2) \geq \f(K_n, F, v-c-2) \geq \f(K_n, T, v-3) \geq \binom{n}{2}/\ex(n,T) = \Omega(n) .\]
\end{proof}

Note that when $H$ is connected, Theorem \ref{lin} gives the linear threshold as $q=e-v+3$. What is the best one can do when $H$ is disconnected? The threshold of $q=e-v+2+c$ given by Theorem \ref{lin} is still correct for forests, since in that case $e-v+2+c = 2$. On the other hand, it is not immediately clear if the linear threshold for $H=2K_3$ is $q=4$ or $q=3$.

We now turn to the quadratic threshold, the smallest $q$ for which $f(K_n, H, q)$ is quadratic in $n$. We can answer this question exactly for some graphs, specifically those with a perfect matching and maximum degree at least $v/2$, where $v$ is the number of vertices:

\begin{prop}\label{exactquad}
Suppose $v(H) = v$, $H$ has a matching of size $\floor{v/2}$, and $H$ has a vertex of degree at least $\floor{v/2}$. Letting $e(H) = e$, we have $f(K_n, H, e-\floor{v/2}+2) = \Theta\left(n^2\right)$ and $f(K_n, H, e-\ceil{v/2}+1) = O\left(n^{2-4/v}\right)$.
\end{prop}

\begin{proof}
The upper bound $f(K_n, H, e-\ceil{v/2}+1) = O\left(n^{2-4/v}\right)$ comes from Theorem \ref{LLL}. The upper bound $f(K_n, H, e-\floor{v/2}+2) = O\left(n^2\right)$ is trivial.

The lower bound $\f(K_n, H, \floor{v/2}-2) = \Omega\left(n^2\right)$ is obtained as follows. Suppose we have a coloring of $K_n$ where every copy of $H$ has at most $\floor{v/2}-2$ repeats. If there is either a monochromatic matching or a monochromatic star with $\floor{v/2}$ edges, then there is a copy of $H$ with at least $\floor{v/2}-1$ repeats. For a given color, the endpoints of a maximal matching form a vertex cover of size less than $v$, each vertex of which is incident to fewer than $\floor{v/2}$ edges. Thus each color class has size less than $v^2/2$, and so there are at least $\frac{2}{v^2} \binom{n}{2}$ colors.
\end{proof}

Note that the proof of Proposition \ref{exactquad} relies on only two forbidden substructures, both of which are monochromatic. The more advanced proof of Pohoata and Sheffer \cite{adam} uses two forbidden substructures, but only one is monochromatic. The other contains many colors and is crucial to obtaining lower bounds of the form $\Omega\left(n^{\alpha}\right)$ for fractional $\alpha$ (for example, in Proposition \ref{pathodd} below).

Proposition \ref{exactquad} clearly generalizes to the following statement.

\begin{prop}\label{quad}
If $H$ has a matching of size $b$ and a vertex of degree at least $b$, then $f(K_n, H, e(H)-b+2) = \Theta\left(n^2\right)$, or equivalently, $\f(K_n, H, b-2) = \Theta\left(n^2\right)$.
\end{prop}

One may hope that the quadratic threshold for $H$ is always the one given in Proposition \ref{exactquad}, $q = e-\floor{v/2}+2$, even when the hypotheses are not satisfied. In fact, this is the case for paths (see Theorem \ref{maintheorem}). However, there are three extreme cases where the quadratic threshold does not even exist: a star with $t$ edges, denoted $S_t$, a matching with $t$ edges, denoted $t K_2$, and the triangle. In the following it is understood that $t \geq 2$. Recall that $\f(K_n, H, 0) = f(K_n, H, e(H))$ is the minimum number of colors needed so that every copy of $H$ is `rainbow.'

\begin{coloring}\label{order}
Label the vertices $1, \dots, n$ and color the edge between $i$ and $j$ with color $i$ if $i < j$. This shows $\f(K_n, tK_2, 0) = O(n)$.
\end{coloring}

In fact, if we color the edge between vertex $n-1$ and vertex $n$ with color $n-2$, this gives a coloring showing that $\f(K_n, tK_2, 0) \leq n-2$. To see that $\f(K_n, tK_2, 0) \geq n-2$, note that every color class must either be a star or a triangle. The number of edges covered by $k$ stars is at most $(n-1) + \cdots + (n-k)$, so the number of edges covered by $n-3$ of these color classes is at most $(n-1) + \cdots + 3 < \binom{n}{2}$.

\begin{coloring}\label{fact}
The edge set of $K_n$ can be partitioned into matchings of size $\floor{n/2}$ (when $n$ is even this partition is known as a $1$-factorization). Color the edges of $K_n$ according to which matching contains them. This coloring contains either $n$ or $n-1$ colors, and shows that $\f(K_n, S_t, 0) = O(n)$.
\end{coloring}

This coloring actually shows that $\f(K_n, S_t, 0) = \f(K_n, K_3, 0) = 2\ceil{n/2}-1$, since each color class must be a matching, and in a $1$-factorization all the color classes attain their maximum size. Similarly, we have $\f(K_n, K_3, 0) = 2\ceil{n/2}-1$.

\begin{cor}\label{cor}
If $\f(K_n, H, 0) = O\left(n^{2-\ep}\right)$ for some $\ep>0$, then $H$ is either a matching, a star, or a triangle. In that case, $\f(K_n, H, 0) = \Theta(n)$.
\end{cor}

\begin{proof}
If $H$ has a matching of size two and a vertex of degree at least two, then $\f(K_n, H, 0) = \Theta\left(n^2\right)$ by Proposition \ref{quad}. The only graphs without those properties are matchings, stars, and the triangle.
\end{proof}

\section{Paths}

We now turn our attention to the asymptotics of $f(K_n, P_v, q)$, where $P_v$ is a path on $v \geq 4$ vertices. (Recall that the asymptotics of $P_3 = S_2$ were found in Section 2.)  We show the range of $q$ for which $f(K_n, P_v, q)$ is linear in $n$ in Proposition \ref{pathlin}, and the range for which it is quadratic in $n$ in Theorem \ref{pathquad}. The only gap in these ranges is bounded in Proposition \ref{pathodd}. See Figure \ref{results} for a summary of these results.

\begin{prop}\label{pathlin}
For every $2 \leq q \leq \floor{v/2}$, we have $f(K_n, P_v, q) = \Theta(n)$.
\end{prop}

\begin{proof}
\[ \Omega(n) = f(K_n, P_v, 2) \leq f(K_n, P_v, q) \leq f(K_n, P_v, \floor{v/2}) = O(n). \]
The lower bound follows by Theorem \ref{lin} and the upper bound by Coloring \ref{order}.
\end{proof}

Now we find the quadratic threshold for paths. We first give a lemma which facilitates the proofs of Theorem \ref{pathquad} and Proposition \ref{pathodd}.

\begin{lemma}\label{pathlemma}
Fix $v \geq 3$. In any coloring of $K_n$, one of three things must occur:
\begin{itemize}
\item there exists a copy of $P_v$ with at least $\floor{(v-1)/2}$ repeats,
\item $\Omega(n^2)$ colors are used, or
\item there exist at least $n^2/16k\log n$ color classes each with a matching of at least $k/16v$ edges, for some $k \geq 16v^2$.
\end{itemize}
\end{lemma}

\begin{proof}

Let $K_n$ be colored with colors from $C$ such that every copy of $P_v$ repeats at most $\ceil{(v-1)/2}-1$ colors. First note that every color class has size less than $vn$, or else we would find a monochromatic $P_v$ (see e.g.~\cite{gallai}). For a color $c$, let $E_c$ denote the set of edges of color $c$.

Suppose that for every $k \geq 16v^2$, there are at most $n^2/8k \log n$ color classes each of size at least $k$. Let $C_i = \{c : i \leq |E_c| \}$, and let $C' = \{c : |E_c| \leq 32 v^2\}$. Then
\[ \binom{n}{2} = \sum_{i=0}^{\floor{\log (vn)}} \sum_{c \in C_{2^{i}}\setminus C_{2^{i+1}}} |E_c| \leq |C'| 32v^2 + \sum_{i=\floor{\log (32v^2)}}^{\floor{\log (vn)}} 2^{i+1} \cdot \frac{n^2}{2^i \cdot 8 \log n} \leq |C'| 32v^2 + \frac{n^2}{4} \]
which implies $|C'| = \Omega(n^2)$, so $|C| = \Omega(n^2)$, as desired.

Now suppose that there is some $k \geq 16v^2$ such that $|C_k| \geq n^2/8k\log n$; fix one such $k$. Let $C^*$ be the set of colors $c \in C_k$ such that there is a monochromatic matching in color $c$ with at least $\frac{k}{16v}$ edges. If $|C^*| \geq \frac{1}{2} |C_k| \geq n^2/16k\log n$, then we are done. Otherwise, we have $|C^*| < \frac{1}{2} |C_k|$, and we will find a copy of $P_v$ with at least $\floor{(v-1)/2}$ repeats with the use of an auxiliary bipartite graph.

Let $c \in C_k \setminus C^*$, and let $S$ be a minimum vertex cover for the edges of color $c$. Since the endpoints of the edges of a maximal matching form a vertex cover and $c$ has a maximal matching of size at most $\frac{k}{16v}$, we have $|S| \leq \frac{k}{8v}$. Since the vertices of $S$ of $c$-degree less than $4v$ cover at most $4v |S| \leq \frac{k}{2}$ edges of $E_c$, then at least half of $E_c$ is incident to vertices of $c$-degree at least $4v$.

Let $G$ be the bipartite graph with parts $X = V(K_n)$ and $Y = \{(u,c): c \in C_k \setminus C^*, u \in V(K_n), d_c(u) \geq 4v \}$ and an edge between $x \in X$ and $(u,c) \in Y$ if the edge $ux$ is colored with $c$. Note that for a given color $c$, the number of edges incident to vertices of the form $(u,c)$ is at least $\frac{1}{2}|E_c| \geq \frac{k}{2}$. This implies
\[ |E(G)| \geq \frac{k}{2}|C_k\setminus C^*| \geq \frac{k}{4}|C_k| \geq \frac{n^2}{32\log n} .\]
On the other hand, it is clear from the definition of $Y$ that $|E(G)| \geq 4v|Y|$.

We delete vertices of degree at most $v$ repeatedly from $X$ and $Y$ until we cannot delete any more, leaving us with a bipartite graph $G'$ with minimum degree greater than $v$. To show that $G'$ is not empty, suppose we delete $D$ edges from $G$; observe that
\[ |E(G')| = |E(G)| - D \geq \max\left(4v|Y|, \frac{n^2}{32\log n}\right) - v (|Y| + |X|) > 0 ,\]
which can be seen by splitting into the cases $|Y| > |X|$ and $n = |X| \geq |Y|$. Thus $G'$ is not empty, and the minimum degree of $G'$ is at least $v+1$.

We greedily construct a path $x_1 (u_1, c_1) \cdots x_{\floor{v/2}} (u_{\floor{v/2}}, c_{\floor{v/2}}) x_{\floor{v/2}+1}$ in $G'$ as follows: select $x_i$ not equal to any of $x_1,\dots,x_{i-1}$ or $u_1,\dots, u_{i-1}$, and select $u_i$ not equal to any of $x_1,\dots,x_i$ or $u_1,\dots,u_{i-1}$. When selecting $x_i$, only $2i-2 < v+1$ vertices of $X$ have been forbidden, so we may choose greedily. When selecting $u_i$, it is possible that we have forbidden many vertices of $Y$, since many vertices of $Y$ correspond to the same $u_j$. However, the vertex $x_j$ is adjacent to $(u_j, c_j)$ only if the edge $x_j u_j$ has color $c_j$, and so $x_j$ has at most one neighbor in $Y$ for every choice of $u_j$. Thus at most $2i-1 < v+1$ of $x_i$'s neighbors in $Y$ are forbidden, so we may choose $u_i$ greedily.

If $v$ is even, this allows us to construct the path $x_1 u_1 \cdots x_{v/2} u_{v/2}$, which has $v/2 - 1 = \floor{(v-1)/2}$ repeats. If $v$ is odd, we construct the path $x_1 u_1 \cdots x_{(v-1)/2} u_{(v-1)/2} x_{(v+1)/2}$, which has $(v-1)/2$ repeats (see Figure \ref{case1}).

\end{proof}

\begin{figure}
\centering
\caption{Path found in Case 1 of the proof of Theorem \ref{pathquad} for $v=7$.}
\label{case1}
\begin{tikzpicture}
\draw (0,0) node{\textbullet} node[below]{$x_1$} -- (1,1) node{\textbullet} node[above]{$u_1$} node[midway,left]{$c_1$} -- (2,0) node{\textbullet} node[below]{$x_2$} node[midway,left]{$c_1$} -- (3,1) node{\textbullet} node[above]{$u_2$} node[midway,left]{$c_2$} -- (4,0) node{\textbullet} node[below]{$x_3$} node[midway,left]{$c_2$} -- (5,1) node{\textbullet} node[above]{$u_3$} node[midway,left]{$c_3$} -- (6,0) node{\textbullet} node[below]{$x_4$} node[midway,left]{$c_3$};
\end{tikzpicture}
\end{figure}
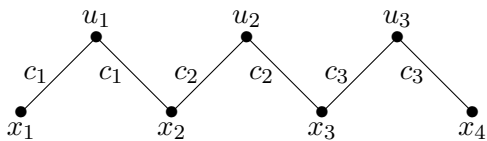

We can think of the argument with the auxiliary bipartite graph as a kind of Tur\'an problem: the auxiliary graph $G$ is reflective of a hypergraph on $V(K_n)$ (the edges are exactly those $c$-neighborhoods of $u$), and we wish to find a kind of `self-avoiding' Berge-path in this hypergraph given that the number of hyperedges present is large.

With Lemma \ref{pathlemma}, we find the quadratic threshold for paths with an even number of vertices in the following theorem.

\begin{theorem}\label{pathquad}
For every $\ceil{v/2}+1 \leq q \leq v-1$, we have $f(K_n, P_v, q) = \Theta\left(n^2\right)$.
\end{theorem}

\begin{proof}

It suffices to show  $\f(K_n, P_v, \ceil{(v-1)/2}-2) = \Omega\left(n^2\right)$. Let $K_n$ be colored with colors from $C$ such that every copy of $P_v$ repeats at most $\ceil{(v-1)/2}-2$ colors. The first case of Lemma \ref{pathlemma} does not occur, and in the second case we are done. If the third case of Lemma \ref{pathlemma} occurs, then we have a monochromatic matching of size $k/16v \geq \floor{v/2}$, with which we have a copy of $P_v$ with $\ceil{(v-1)/2}-1$ repeats (see Figure \ref{matching}). So this case is impossible as well.

\end{proof}

\begin{figure}
\centering
\caption{Edges of a matching strung into a path, as in Case 2 of the proof of Theorem \ref{pathquad}.}
\label{matching}
\begin{tikzpicture}
\draw (0,0) node{\textbullet} -- (0,1) node{\textbullet} node[midway,left]{$c$};
\draw[dashed] (0,1) -- (1,1) node{\textbullet};
\draw (1,1) -- (1,0) node{\textbullet} node[midway,left]{$c$};
\draw[dashed] (1,0) -- (2,0) node{\textbullet};
\draw (2,0) -- (2,1) node{\textbullet} node[midway,left]{$c$};
\draw[dashed] (2,1) -- (3,1) node{\textbullet};
\draw (3,1) -- (3,0) node{\textbullet} node[midway,left]{$c$};
\end{tikzpicture}
\end{figure}
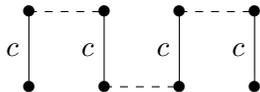

There is a small gap in what we have shown when $v$ is odd, namely $f(K_n, P_v, (v+1)/2)$ has only been shown to be $\Omega(n)$ (and $O(n^2)$ trivially). Note that $v=3$ is equivalent to the star $S_2$, so by Corollary \ref{cor}, $f(K_n, P_3, 2) = \Theta(n)$. The $v=5$ case was shown to be linear by a construction of Rosta \cite{rosta}: label the vertices of $K_n$ with distinct binary strings of length $\ceil{\log n}$, and color the edge between vertices labeled $x$ and $y$ with color $x \oplus y$ (where $\oplus$ is bitwise exclusive-or). This coloring is proper, since $x \oplus y = x \oplus z$ implies $y=z$. In this coloring, if any $P_5$ with vertices $a,b,c,d,e$ (in that order) contains only two colors, it must be that $a \oplus b = c \oplus d$ and $b \oplus c = d \oplus e$. Upon `adding' these equations, we get $a \oplus c = c \oplus e$, meaning $a = e$, contradicting this being a path. Thus Rosta's construction shows $f(K_n, P_5, 3) \leq 2^{\ceil{\log n}} \leq 2n$.

We only have loose bounds when $v \geq 7$. First we need a combinatorial lemma of Erd\H{o}s \cite{erdoslemma}, which we state here for a particular case.

\begin{lemma}\label{erdos1}
Let $A_1, \dots A_N$ be $N$ subsets of $A$, where $|A| = n$ and $|A_i| \geq d\sqrt{n}$ for some constant $d$. If $N \geq \frac{8}{d} \sqrt{n}$, then there are some $i \neq j$ such that $|A_i \cap A_j| \geq \frac{d^2}{2}$.
\end{lemma}

\begin{prop}\label{pathodd}
For odd $v\geq 7$, we have $f(K_n, P_v, \frac{v+1}{2}) = O\left(n^{2-2/(v-1)}\right)$. For odd $v \geq 9$, we have $f(K_n, P_v, \frac{v+1}{2}) = \Omega\left(n^{3/2}/\log n\right)$, and for $v=7$, we have $f(K_n, P_7, 4) = \Omega\left(n^{4/3}/\log^{2/3} n\right)$.
\end{prop}

\begin{proof}

The upper bound comes from Theorem \ref{LLL}.

Let $K_n$ be colored with colors from $C$ so that every $P_7$ contains at least $4$ colors (and at most $2$ repeated colors). By Lemma \ref{pathlemma}, we either have a copy of $P_7$ with at least $3$ repeats (which does not occur), we use $\Omega(n^2)$ colors (so we are done), or there exists a $k \geq 16v^2$ such that there are at least $n^2/16k\log n$ color classes each with a matching of at least $k/16v$ edges. Let $C^*$ be the set of these colors. If $k \leq n^{2/3}/\log^{1/3} n$, then there are already at least $\Omega(n^{4/3}/\log^{2/3} n)$ many colors. Otherwise, $k \geq n^{2/3}/\log^{1/3} n$. For $c \in C^*$, there is a matching $M$ of size $ k/16v \geq n^{2/3}/(16v\log^{1/3} n)$ in color $c$; let $S$ be the matched vertices of $M$. If an edge of color $c$ appears between vertices of $S$ but is not in the matching $M$, then we have a $P_7$ with $3$ repeats, which is not allowed. If there are five edges of color $c'$ among the vertices of $S$, these edges must be incident to at least three edges of $M$, and so we may find a $P_7$ with $3$ repeats as well (Figure \ref{P7} enumerates these possibilities and shows the desired $P_7$ in each case). Therefore there are at least $\Omega(n^{4/3}/\log^{2/3} n)$ colors (even just among the vertices of $S$).

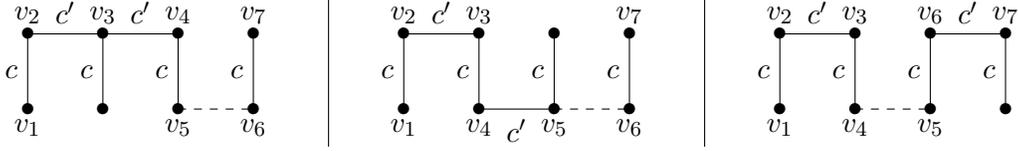
\begin{figure}
\centering
\caption{Cases with a large matching when $v=7$ in Proposition \ref{pathodd}}
\label{P7}
\begin{tikzpicture}
\draw (0,0) node{\textbullet} -- (0,1) node{\textbullet} node[midway,left]{$c$};
\draw (1,0) node{\textbullet} -- (1,1) node{\textbullet} node[midway,left]{$c$};
\draw (2,0) node{\textbullet} -- (2,1) node{\textbullet} node[midway,left]{$c$};
\draw (3,0) node{\textbullet} -- (3,1) node{\textbullet} node[midway,left]{$c$};
\draw (0,1) -- (1,1) node[midway,above]{$c'$};
\draw (1,1) -- (2,1) node[midway,above]{$c'$};
\draw (0,0) node[below]{$v_1$};
\draw (0,1) node[above]{$v_2$};
\draw (1,1) node[above]{$v_3$};
\draw (2,1) node[above]{$v_4$};
\draw[dashed] (2,0) node[below]{$v_5$} -- (3,0) node[below]{$v_6$};
\draw (3,1) node[above]{$v_7$};

\draw (4,-.5) -- (4,1.5);
\begin{scope}[shift={(5,0)}]
\draw (0,0) node{\textbullet} -- (0,1) node{\textbullet} node[midway,left]{$c$};
\draw (1,0) node{\textbullet} -- (1,1) node{\textbullet} node[midway,left]{$c$};
\draw (2,0) node{\textbullet} -- (2,1) node{\textbullet} node[midway,left]{$c$};
\draw (3,0) node{\textbullet} -- (3,1) node{\textbullet} node[midway,left]{$c$};
\draw (0,1) -- (1,1) node[midway,above]{$c'$};
\draw (1,0) -- (2,0) node[midway,below]{$c'$};
\draw (0,0) node[below]{$v_1$};
\draw (0,1) node[above]{$v_2$};
\draw (1,1) node[above]{$v_3$};
\draw (1,0) node[below]{$v_4$};
\draw[dashed] (2,0) node[below]{$v_5$} -- (3,0) node[below]{$v_6$};
\draw (3,1) node[above]{$v_7$};
\end{scope}

\draw (9,-.5) -- (9,1.5);
\begin{scope}[shift={(10,0)}]
\draw (0,0) node{\textbullet} -- (0,1) node{\textbullet} node[midway,left]{$c$};
\draw (1,0) node{\textbullet} -- (1,1) node{\textbullet} node[midway,left]{$c$};
\draw (2,0) node{\textbullet} -- (2,1) node{\textbullet} node[midway,left]{$c$};
\draw (3,0) node{\textbullet} -- (3,1) node{\textbullet} node[midway,left]{$c$};
\draw (0,1) -- (1,1) node[midway,above]{$c'$};
\draw (2,1) -- (3,1) node[midway,above]{$c'$};
\draw (0,0) node[below]{$v_1$};
\draw (0,1) node[above]{$v_2$};
\draw (1,1) node[above]{$v_3$};
\draw (2,1) node[above]{$v_6$};
\draw[dashed] (1,0) node[below]{$v_4$} -- (2,0) node[below]{$v_5$};
\draw (3,1) node[above]{$v_7$};
\end{scope}
\end{tikzpicture}
\end{figure}

To show the lower bound when $v \geq 9$, again apply Lemma \ref{pathlemma}: the first case does not occur, and in the second case we are done, so suppose we are in the third case as before. If $k\leq 32 v^{3/2} n^{1/2}$ we have $\Omega(n^{3/2}/\log n)$ many colors in $C^*$ and we are done. So suppose $k \geq 32 v^{3/2} n^{1/2}$. Let $C^* = \{c_1, \dots, c_{N}\}$, let $M_i$ be a maximum matching in color $c_i$, and let $d=4\sqrt{v}$. Note that $|V(M_i)| \geq 2 \frac{k}{16v} \geq d \sqrt{n}$ and $N = |C^*| \geq \frac{n^2}{16k \log n} \geq \frac{8}{d} \sqrt{n}$, the latter of which follows because $k \leq vn$, since no color class has size more than $vn$. We have by Lemma \ref{erdos1} that there are two indices $i$ and $j$ such that $|V(M_i) \cap V(M_j)| \geq \frac{d^2}{2} = 8v$. The graph whose edges are the edges of $M_i$ and $M_j$ incident to at least one vertex of $V(M_i) \cap V(M_j)$ contains a disjoint union of even cycles and paths of length at least two on at least $8v > \frac{5}{3} v$ vertices. Thus we can find $\ceil{\frac{v}{3}}$ disjoint copies of an edge of color $c_i$ incident with an edge of color $c_j$, as in Figure \ref{cherry}. Stringing these together into a path with $v$ vertices gives at least $\floor{\frac{2}{3}v}-2 \geq \frac{v-1}{2}$ many repeated colors (with equality when $v = 9$, 11, and 13), which is not allowed.

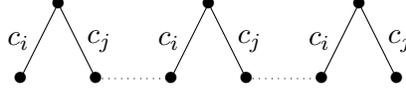
\begin{figure}
\centering
\caption{Multicolored stars that are strung into a path.}
\label{cherry}
\begin{tikzpicture}
\draw (0,0) node{\textbullet} -- (.5,1) node{\textbullet} node[midway,left]{$c_i$};
\draw (.5,1) -- (1,0) node{\textbullet} node[midway,right]{$c_j$};
\draw[dotted] (1,0) -- (2,0) node{\textbullet};
\draw (2,0) -- (2.5,1) node{\textbullet} node[midway,left]{$c_i$};
\draw (2.5,1) -- (3,0) node{\textbullet} node[midway,right]{$c_j$};
\draw[dotted] (3,0) -- (4,0) node{\textbullet};
\draw (4,0) -- (4.5,1) node{\textbullet} node[midway,left]{$c_i$};
\draw (4.5,1) -- (5,0) node{\textbullet} node[midway,right]{$c_j$};
\end{tikzpicture}
\end{figure}

\end{proof}

\section{Acknowledgements}

This research project was conducted at 2018 CUNY Combinatorics REU. Many thanks go to Adam Sheffer, who ran the program, fostered discussions, and had several helpful comments on this paper. Additional thanks go to a referee for suggestions which improved the clarity of the proofs.

\bibliographystyle{abbrv}
\bibliography{references}

\end{document}